\documentclass{amsart}
\usepackage[latin1]{inputenc}
\usepackage[english]{babel}
\usepackage{csquotes}
\usepackage{todonotes}

\usepackage{color} 
\usepackage[colorlinks=true,linktoc=all,linkcolor=blue,citecolor=red,filecolor=pink,urlcolor=blue]{hyperref}
\usepackage[hyperref=false,style=alphabetic,citestyle=alphabetic,backend=bibtex,maxnames=100,doi=false,isbn=false,url=false,giveninits=true]{biblatex}
\bibliography{bibli}

\usepackage{amssymb,amsmath,latexsym}

\theoremstyle{plain}
\newtheorem{theorem}{Theorem}[section]
\newtheorem{prop}[theorem]{Proposition}
\newtheorem{lemma}[theorem]{Lemma}
\newtheorem{defin}[theorem]{Definition}
\newtheorem{oss}[theorem]{Remark}

\numberwithin{equation}{section}

\begin{document}
\title{Extending an example by Colding and Minicozzi}
\author{Lorenzo Ruffoni}
\address{Mathematics Department, Florida State University, Tallahassee FL 32306, USA}
\email{lruffoni@fsu.edu}

\author{Francesca Tripaldi}
\address{Department of Mathematics and Statistics, University of Jyv\"{a}skyl\"{a}, Jyv\"{a}skyl\"{a} FI-40014, Finland}
\email{francesca.f.tripaldi@jyu.fi}

\subjclass[2010]{53A10, 49Q05}

\begin{abstract}
Extending an example by Colding and Minicozzi \cite{cm28}, we construct a sequence of properly embedded minimal disks $\Sigma_i$ in an infinite Euclidean cylinder around the $x_3$-axis with curvature blow-up at a single point. The sequence converges to a non smooth and non proper minimal lamination in the cylinder. Moreover, we show that the disks $\Sigma_i$ are not  properly embedded in a sequence of open subsets of $\mathbb R^3$ that exhausts $\mathbb R^3$.
\end{abstract}

 \maketitle
 \tableofcontents

\section{Introduction}
In a series of influential papers \cite{cm21,cm22,cm23,cm24}, Colding and Minicozzi initiated the study of sequences of minimal disks in a 3-manifold. In general, if no restriction on the curvatures of the disks is required, it is known that some wild behaviour should be expected. For instance, in \cite{cm28} Colding and Minicozzi constructed an example of a sequence of minimal disks in a Euclidean ball for which curvatures blow up at the centre of the ball, and such that the limit lamination is neither smooth nor proper. This result is obtained by a careful analysis of the Weierstrass representation of the minimal disks.

Following this example, a number of similar results have been obtained via analogous methods, in which comparable wild limits are observed with the curvature blowing up at a finite set of points on a line \cite{Dean2006}, along a closed segment \cite{Khan}, or more generally any compact subset of a line \cite{Kleene}. Using variational methods, Hoffman and White \cite{HoffmanWhite1} provided examples of minimal disks in an infinite Euclidean cylinder with curvature blow-ups along any prescribed compact subset of the axis of the cylinder. 

In this paper, we follow the approach of \cite{cm28} and show that their example can be extended to create a sequence of properly embedded minimal disks $\Sigma_i$ in an infinite Euclidean cylinder, displaying the same pathologies.
By \cite{cm24}, it is known that these pathologies do not occur if the sequence $\Sigma_i$ is indeed a sequence of minimal disks properly embedded in a sequence of open sets which invade the whole $\mathbb R^3$.

Here we prove that the disks $\Sigma_i$, which we obtained by extending the ones constructed in \cite{cm28}, do not give rise to minimal disks properly embedded in a growing family of cylinders which exhaust $\mathbb R^3$.

Our main result is therefore analogous to Theorem 1 in \cite{cm28} and can be stated as follows; we refer to \cite{cm28} for the definition of multi-valued graph and pictures.
\begin{theorem}\label{main thm}
One can construct a sequence of properly embedded minimal disks $0\in\Sigma_i\subset\lbrace x_1^2+x_2^2\le 1, x_3\in\mathbb R\rbrace\subset\mathbb R^3$  containing the $x_3$-axis, $\lbrace(0,0,t)\mid t\in\mathbb R\rbrace\subset\Sigma_i$, and such that the following conditions are satisfied:
\begin{itemize}
    \item [(1)] $\lim_{i\to\infty}\vert A_{\Sigma_i}\vert^2(0)=\infty$;
    \item [(2)] $\sup_i\sup_{\Sigma_i\setminus B_\delta}\vert A_{\Sigma_i}\vert^2<\infty$ for all $\delta>0$;
    \item [(3)] $\Sigma_i\setminus\lbrace x_3\text{-axis}\rbrace=\Sigma_{1,i}\cup\Sigma_{2,i}$, for multi-valued graphs $\Sigma_{1,i}$ and $\Sigma_{2,i}$;
    \item [(4)] $\Sigma_{i}\setminus \lbrace x_3=0\rbrace$ converges to two embedded minimal disks $\Sigma^{\pm}\subset\lbrace\pm x_3>0\rbrace$ with $\overline{\Sigma^{\pm}}\setminus\Sigma^{\pm}=\lbrace x_1^2+x_2^2\le 1, x_3=0\rbrace$. Moreover, $\Sigma^{\pm}\setminus\lbrace x_3\text{-axis}\rbrace=\Sigma^{\pm}_1\cup\Sigma^{\pm}_2$ for multi-valued graphs $\Sigma^\pm_1$ and $\Sigma^\pm_2$ each of which spirals into $\lbrace x_3=0\rbrace$.
\end{itemize}

\end{theorem}

From $(4)$ we get that $\Sigma_i\setminus\lbrace 0\rbrace$ converges to a minimal lamination of $\lbrace x_1^2+x_2^2\le 1,x_3\in\mathbb R\rbrace\setminus\lbrace 0\rbrace$ (with leaves $\Sigma^-$, $\Sigma^+$, and $\lbrace x_1^2+x_2^2\le 1, x_3=0\rbrace\setminus\lbrace 0\rbrace$) which does not extend to a lamination of $\lbrace x_1^2+x_2^2\le 1, x_3\in\mathbb R\rbrace$. In other words, $0$ is not a removable singularity.

The structure of the paper is as follows: in Section 2 we review the definitions and basic results on the Weierstrass representation of minimal surfaces in the Euclidean space $\mathbb R^3$, and set up some notations for the domains and functions that will be used. In Section 3, we show that the disks constructed with those data are properly embedded in a fixed infinite Euclidean cylinder around the $x_3$-axis, and in Section 4 we show that they are not a sequence of properly embedded disks in an exhausting family. Section 5 contains the conclusion of the proof of the main theorem.\par

\vspace{.5cm}
\textbf{Acknowledgements}: the authors wish to thank the Department of Mathematics of Bologna University for its kind hospitality during the early stage of this project. This work has been partially supported by the Academy of Finland (grant 288501 `\emph{Geometry of subRiemannian groups}'), by the European Research Council  (ERC Starting Grant 713998 GeoMeG `\emph{Geometry of Metric Groups}') and by the European Union's Horizon 2020 research and innovation programme under the Marie Sk\l{}odowska-Curie grant agreement No 777822 (`\emph{Geometric and Harmonic Analysis with Interdisciplinary Applications}').



\section{Preliminaries and set up}

Let us first fix some notation. Following \cite{cm28}, we will use $(x_1,x_2,x_3)$ as the coordinates in $\mathbb R^3$ and $z=x+iy$ in $\mathbb C$. Given $f:\mathbb C\to\mathbb C^n$, $\partial_xf$ and $\partial_yf$ will denote $\frac{\partial f}{\partial x}$ and $\frac{\partial f}{\partial y}$ respectively, and likewise we will have $\partial_zf=\frac{1}{2}(\partial_xf-i\partial_yf)$. Given a point $p=(p_1,p_2,p_3)\in\mathbb R^3$ and a value $r>0$, we will denote by $D_{r}(p)$ the Euclidean disk of radius $r$, contained in the horizontal plane $\lbrace x_3=p_3\rbrace$, and centred at the point $p=(p_1,p_2,p_3)$.

When $\Sigma$ is a smooth surface embedded in $\mathbb R^3$, we denote by $\mathbf{K}_\Sigma$ its sectional curvature, and by $A_\Sigma$ its second fundamental form; so when $\Sigma$ is minimal we will have $\vert A_\Sigma\vert^2=-2 \mathbf{K}_\Sigma$. Moreover, when $\Sigma$ is oriented, $\mathbf{n}_\Sigma$ will denote its unit normal.

Let us recall the classical definition of a Weierstrass representation and its role in constructing minimal surfaces (see \cite{osserman2002survey}). Let $\Omega\subset\mathbb C$ be a domain in the complex plane, $g$ a meromorphic function on $\Omega$ and $\phi$ a holomorphic 1-form on $\Omega$; then we refer to the couple $(g,\phi)$ as a Weierstrass representation.

Given a a Weierstrass representation $(g,\phi)$, one can associate to it a conformal minimal immersion $F:\Omega\to\mathbb R^3$
\begin{align}\label{ImmersionF}
  F(z)=Re\int_{\zeta\in\gamma_{z_0,z}}\bigg(\frac{1}{2}\big(g^{-1}(\zeta)-g(\zeta)\big), \frac{i}{2}\big(g^{-1}(\zeta)+g(\zeta)\big),1\bigg)\phi(\zeta)\,.
\end{align}

Here, $z_0\in\Omega$ denotes a fixed base point, and the integral is taken along a path $\gamma_{z_0,z}$ that joins $z_0$ to $z$ in $\Omega$. As soon as  $g$ has no zeros or poles and $\Omega$ is simply connected, $F$ will not depend on the choice of the integration path (see \cite{cm28}). On the other hand the choice of the base point $z_0$ will change the value of $F$ by an adding constant.

Using the Weierstrass data, one can find explicit formulae for the unit normal $\mathbf{n} $ and the Gauss curvature $\mathbf{K}$ (see Sections 8 and 9 in ~\cite{osserman2002survey}):
\begin{align}
    \label{unit normal 1} \mathbf{n}&=\frac{\big(2 Re g,2 Im g, \vert g\vert^2-1\big)}{\vert g\vert^2+1}\,,\\
    \mathbf{K}&=-\bigg[\frac{4\vert\partial_zg\vert\,\vert g\vert}{\vert\phi\vert\,(1+\vert g\vert^2)^2}\bigg]^2
\end{align}

As previously pointed out, the main goal of this paper is to provide an unbounded extension of the domain, on which to keep the same choice of the Weierstrass data as in \cite{cm28}. The map $F$ is guaranteed to be an immersion (that is, $dF\neq 0$) as soon as $\phi$ does not vanish and $g$ has no zeros or poles.

We include the following lemma (see Lemma 1 in \cite{cm28}), since its results will be useful for some further computations later on.

\begin{lemma}\label{lemma_derF}
Let $F$ be as in equation (\ref{ImmersionF}) for the Weierstrass data $(g,\phi)$ given by  $g(z)=e^{i(u(z)+iv(z)}$ and $\phi=dz$, then
\begin{align}
   \label{dxF} \partial_xF&=(\sinh v\cdot\cos u, \sinh v\cdot\sin u, 1)\,,\\
  \label{dyF} \partial_yF&=(\cosh v \cdot\sin u, -\cosh v\cdot\cos u,0)\,.
\end{align}
\end{lemma}

Let us construct the one-parameter family (with parameter $a\in (0,1/2]$) of minimal immersions $F_a$ that we will be using to construct our family of embedded minimal disks $\Sigma_a$.

\begin{defin} \label{ChoiceWeierstrassData} We choose the same Weierstrass data $(g_a,\phi)$ as in \cite{cm28}, that is:
\begin{align}
    g_a=e^{ih_a},\text{ where } h_a(z)=\frac{1}{a}\arctan \bigg(\frac{z}{a}\bigg)\text{ and }
    \phi=dz\,.
\end{align}

In this case, the two real functions $u$ and $v$ in Lemma \ref{lemma_derF} will then depend on the parameter $a$: $h_a=u_a+iv_a$ (see Remark \ref{Explicit formulation for u and v} for an explicit formulation of both $u_a$ and $v_a$).
\end{defin}

In this paper, we are extending the original bounded domain $\Omega_a$ defined in \cite{cm28} to an unbounded one, which will still be denoted by $\Omega_a$. To do so, we will first introduce the two families of domains $\mathcal{D}^1_a$ and $\mathcal{D}^2_a$, as follows (see Figure \ref{fig:domains}):
\begin{align}
    \mathcal{D}^1_a=&\left \lbrace (x,y)\left\vert \vert y \vert\le \frac{(x^2+a^2)^{3/4}}{2}\right.\right \rbrace\,,\\
    \mathcal{D}^2_a=&\left\lbrace (x,y)\left\vert \vert y\vert\le \frac{(x^2+a^2)^{1/2}}{2}\right.\right\rbrace\,.
\end{align}
\begin{center}
\begin{figure}
    \centering
    \includegraphics[width=13.2cm]{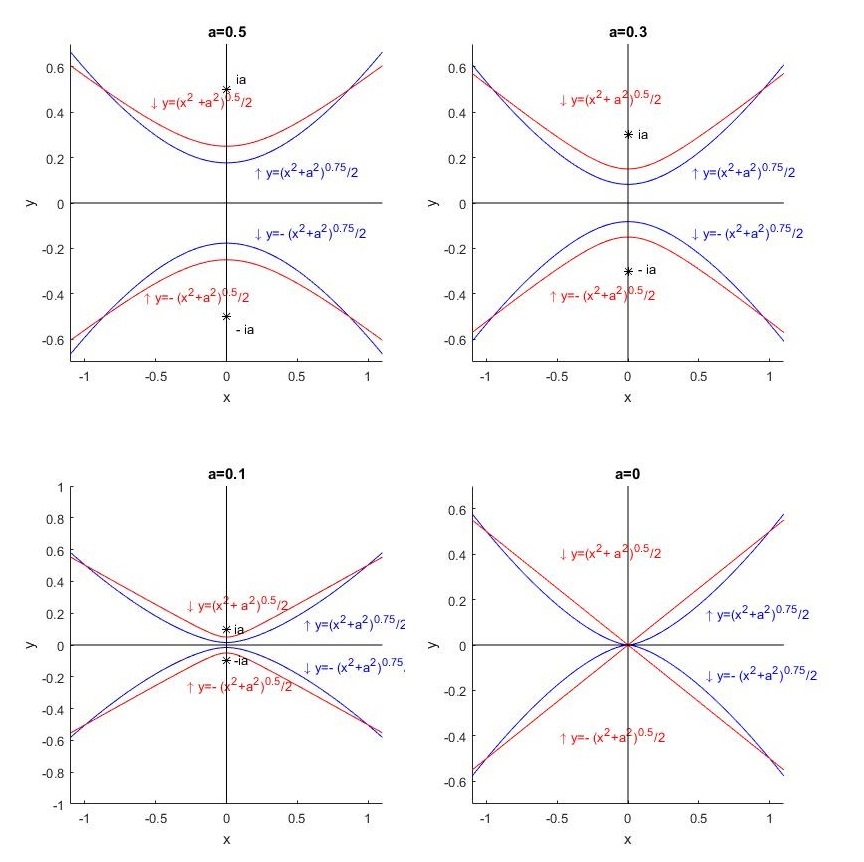}
    \caption{Graphical representations of both domains $\mathcal{D}^1_a$ (in blue) and $\mathcal{D}^2_a$ (in red) for different values of the parameter $a$.}
    \label{fig:domains}
\end{figure}
\end{center}

One should notice that the two domains $\mathcal{D}^1_a$ and $\mathcal{D}^2_a$ cross at the points where 
\begin{align}
    x=\pm\sqrt{1-a^2}\,,
\end{align}
hence somewhere in the intervals $[\sqrt{3}/2, 1]$ and $[-1,-\sqrt{3}/2]$, depending on the parameter $a\in (0,1/2]$. One should also notice that for $x^2>1-a^2$, the second domain $\mathcal{D}^2_a$ (in red) is strictly contained inside the first domain $\mathcal{D}^1_a$ (in blue), that is
\begin{align}
  \mathcal{D}^2_a\cap\lbrace (x,y)\mid x^2>1-a^2\rbrace\subset\mathcal{D}^1_a\cap  \lbrace (x,y)\mid x^2>1-a^2\rbrace\,.
\end{align}

Viceversa, for $x^2<1-a^2$, we have the opposite inclusion:
\begin{align}\label{dentro}
    \mathcal{D}^1_a\cap\lbrace (x,y)\mid x^2<1-a^2\rbrace\subset\mathcal{D}^2_a\cap  \lbrace (x,y)\mid x^2<1-a^2\rbrace\,.
\end{align}

In Colding and Minicozzi's paper, the domains of definition $\Omega_a$ are exactly taken as follows (see equation (2.1) in \cite{cm28}):
\begin{align}
    \left\lbrace (x,y)\left\vert\vert x\vert\le \frac{1}{2},\,\vert y\vert\le\frac{(x^2+a^2)^{3/4}}{2}\right.\right\rbrace=\mathcal{D}^1_a\cap\left\lbrace (x,y)\left\vert\vert x\vert\le\frac{1}{2}\right.\right\rbrace\,,
\end{align}
which always falls into case (\ref{dentro}), since $\frac{1}{4}<1-a^2$ for any $a\in (0,1/2]$.

Let us take into consideration the value $\frac{1}{\pi}$ (please refer to Remark \ref{reason for xA} for the reasoning behind the choice of this value). If we divide the interval $[\frac{1}{\pi},\frac{1}{2}]$ into thirds, we get the following two values:
\begin{align}\label{def xA}
    x_A&=\frac{1}{\pi}+\bigg(\frac{1}{2}-\frac{1}{\pi}\bigg)\cdot\frac{1}{3}=\frac{1}{\pi}+\frac{\pi-2}{6\pi}=\frac{6+\pi-2}{6\pi}=\frac{4+\pi}{6\pi}\,,\text{ and }\\
    x_B&=\frac{1}{\pi}+\bigg(\frac{1}{2}-\frac{1}{\pi}\bigg)\cdot\frac{2}{3}=\frac{1}{\pi}+\frac{\pi-2}{3\pi}=\frac{3+\pi-2}{3\pi}=\frac{1+\pi}{3\pi}\,.
\end{align}

By symmetry, the points $-x_B$ and $-x_A$ will divide into thirds the interval $[-\frac{1}{2},-\frac{1}{\pi}]$.

As already pointed out, $\mathcal{D}^1_a\cap\lbrace(x,y)\mid\vert x \vert\le \frac{1}{2}\rbrace\subset\mathcal{D}^2_a\cap\lbrace(x,y)\mid\vert x \vert\le \frac{1}{2}\rbrace$, which means that if we consider the two boundaries as follows
\begin{align}
    \vert y\vert=\frac{(x^2+a^2)^{3/4}}{2}\,,\text{ for }\vert x\vert\le x_A\,,\text{ and }
    \vert y\vert=\frac{(x^2+a^2)^{1/2}}{2}\,,\text{ for } \vert x\vert\ge x_B
\end{align}
we will see jumps (see Figure \ref{fig:jumps} for a graphical representation of these jumps) of magnitude equal to $y_B-y_A$, where
\begin{align}
    y_A =\frac{(x_A^2+a^2)^{3/4}}{2}\,,\text{ and }y_B=\frac{(x_B^2+a^2)^{1/2}}{2}\,.
\end{align}
In order to join these two domains 
\begin{align}
    \mathcal{D}^1_a\cap\lbrace (x,y)\mid\vert x\vert\le x_A\rbrace\text{ and }\mathcal{D}^2_a\cap\lbrace (x,y)\mid\vert x\vert\ge x_B\rbrace
\end{align}
we can simply take the following line segments

 \[
    \left\{\begin{array}{lr}
           \xi_1(x)=y_A+(y_B-y_A)\frac{x-x_A}{x_B-x_A} \text{ from } (x_A,y_A) \text{ to }(x_B,y_B)\\
         \xi_2(x)=y_A+(y_B-y_A)\frac{x+x_A}{x_A-x_B}\text{ from } (-x_A,y_A) \text{ to }(-x_B,y_B)\\
         \xi_3(x)=-\xi_1(x)=-y_A+(y_A-y_B)\frac{x-x_A}{x_B-x_A}\text{ from } (x_A,-y_A) \text{ to }(x_B,-y_B)\\
         \xi_4(x)=-\xi_2(x)=-y_A+(y_A-y_B)\frac{x+x_A}{x_A-x_B}\text{ from } (-x_A,-y_A) \text{ to }(-x_B,-y_B)
        \end{array}\right.
  \]

Using these segments, we are then able to define our domains of definition $\Omega_a$.

\begin{center}
\begin{figure}
    
\includegraphics[width=14cm]{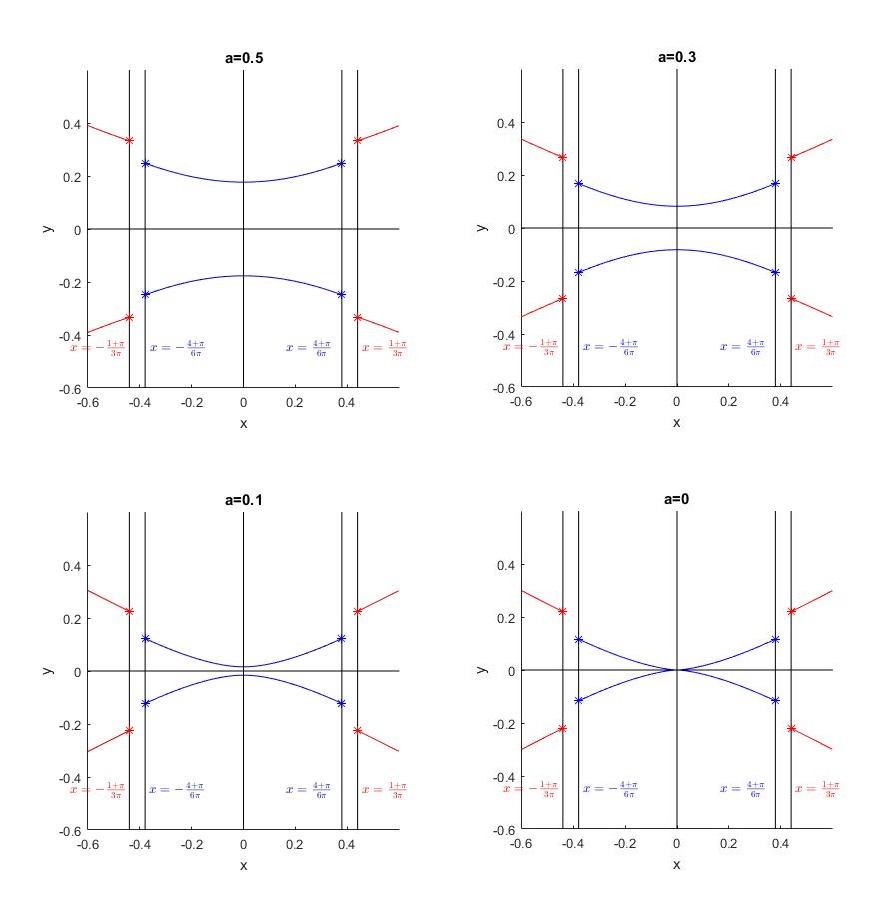}
\caption{Graphical representations of the domains $\mathcal{D}^1_a\cap\lbrace (x,y)\mid \vert x\vert\le x_A\rbrace$ (in blue) and $\mathcal{D}^2_a\cap\lbrace (x,y)\mid\vert x\vert\ge x_B\rbrace$ (in red) for different values of the parameter $a$.}
    \label{fig:jumps}
    \end{figure}
\end{center}

\begin{defin}\label{domains}

Our unbounded domains of definition $\Omega_a$ for $a\in (0,1/2]$ are given by
\[
    \Omega_a=\left\lbrace (x,y) \text{ such that }\left\{\begin{array}{lr}
        \vert y\vert\le \frac{(x^2+a^2)^{3/4}}{2}\text{ for }\vert x\vert \le x_A\\
         \vert y\vert\le \frac{(x^2+a^2)^{1/2}}{2}\text{ for }\vert x\vert \ge x_B\\
        \vert y\vert\le\xi_1(x)\text{ for } x_A\le x\le x_B \\
        \vert y\vert\le\xi_2(x)\text{ for } -x_B\le x\le -x_A
        \end{array}\right.\,\right\rbrace\,,
  \]
 
 and 
 \begin{align}
     \Omega_0=\bigcap_{a>0} \Omega_a\setminus\lbrace 0\rbrace=\Omega_0^{+}\cup\Omega_0^{-}\,,
 \end{align}
 where $\Omega_0^{+}$and $\Omega_0^{-}$ are respectively the right and left components of $\Omega_0$.
 
\end{defin}
\begin{oss}
One should notice that Definition \ref{domains} gives rise to domains $\Omega_a$ whose boundaries are piecewise smooth. In order to obtain smooth domains, it is sufficient to take the convolution of the boundary with a Friedrich mollifier.
\end{oss}
\begin{oss}\label{stima su tutto il dominio}
Let us notice that, by construction, $\Omega_a\subset\mathcal{D}_a^2$ for any $a\in[0,1/2]$.
\end{oss}

\begin{oss}\label{Explicit formulation for u and v}

Let us study the explicit formulation of $h_a=u_a+iv_a$. On a suitable domain of definition we have, using the standard notation $z=x+iy$, that
\begin{align*}
    h_a(z)&=\frac{1}{a}\arctan\bigg(\frac{z}{a}\bigg)=\frac{i}{2a}\bigg[Log\bigg(1-\frac{iz}{a}\bigg)-Log\bigg(1+\frac{iz}{a}\bigg)\bigg]\\=&\frac{i}{2a}\bigg[Log\bigg(\frac{a-ix+y}{a}\bigg)-Log\bigg(\frac{a+ix-y}{a}\bigg)\bigg]\\=&\frac{i}{2a}\bigg[\frac{1}{2}\log\bigg(\frac{(a+y)^2+x^2}{a^2}\bigg)+i Arg\bigg(\frac{a+y-ix}{a}\bigg)+\\&-\frac{1}{2}\log\bigg(\frac{(a-y)^2+x^2}{a^2}\bigg)-i Arg\bigg(\frac{a-y+ix}{a}\bigg)\bigg]\\=&\frac{i}{2a}\bigg[\frac{1}{2}\log\bigg(\frac{(a+y)^2+x^2}{(a-y)^2+x^2}\bigg)+iArg\bigg(\frac{a+y-ix}{a}\bigg)-iArg\bigg(\frac{a-y+ix}{a}\bigg)\bigg]\\=& \frac{1}{2a}Arg\bigg(\frac{a-y+ix}{a}\bigg)-\frac{1}{2a}Arg\bigg(\frac{a+y-ix}{a}\bigg)+\frac{i}{4a}\log\bigg(\frac{(a+y)^2+x^2}{(a-y)^2+x^2}\bigg)\\=&u_a(x,y)+iv_a(x,y)\,. 
\end{align*}

\end{oss}
\begin{oss}
Let us further notice that the functions $h_a$ are indeed well-defined, since the domains $\Omega_a$ are simply connected and the points $\pm ia\notin\Omega_a$. In fact, when $x=0$, we have $\vert y\vert\le a^{3/2}/2$, and we have that $a^{3/2}/2<a$ if and only if $a<4$ (and by construction we are assuming $a\le 1/2$).

\end{oss}

\begin{defin} In order to fix the notation, we will be denoting by $F_a$ the conformal minimal immersion $F_a:\Omega_a\to\mathbb R^3$ associated to the Weierstrass representation $(g_a,\phi)$ for any $a\in (0,1/2]$.
\end{defin}

In the following lemma, we will be covering some explicit expressions and bounds that will be useful in the next sections. Unless otherwise stated, from now on we will be working within our domains of definition, that is $z=x+iy\in\Omega_a$.

\begin{lemma}\label{easybounds}
Let $h_a(z)=u_a(z)+iv_a(z)=\frac{1}{a}\arctan\left(\frac{z}{a}\right)$ be as above; then the following holds:
\begin{enumerate}
    \item $\partial_z h_a(z)=\dfrac{1}{z^2+a^2}=\dfrac{x^2+a^2-y^2-2ixy}{(x^2+a^2-y^2)^2+4x^2y^2}\,,$
    \item $\mathbf{K}_a(z)=-\dfrac{\vert\partial_z h_a\vert^2}{\cosh^4(v_a)}=-\dfrac{\vert z^2+a^2\vert^{-2}}{\cosh^4(Im(\arctan(z/a))/a)}\,,$
    \item $\partial_y u_a(z)=\dfrac{2xy}{(x^2+a^2-y^2)^2+4x^2y^2}\,,$ 
    \item $\partial_y v_a(z)=\dfrac{x^2+a^2-y^2}{(x^2+a^2-y^2)^2+4x^2y^2}\,,$

\item $|\partial_y u_a(z)| \leq \dfrac{4|xy|}{(x^2+a^2)^2}\,, $
\item $\partial_y v_a(z) > \dfrac{3}{8(x^2+a^2)}\,,$
\item $v_a(x,y)\geq 0\text{ for }y\ge 0\text{ and }v_a(x,y)<0\text{ for }y<0\,.$
\end{enumerate}
\end{lemma}
\begin{proof}
(1) and (2) are just computations, and (3) and (4) follow directly from (1) by applying the Cauchy-Riemann equations to $h_a$ (see equations (2.2), (2.3) and (2.4) in \cite{cm28}).

Both inequalities (5-6) are obtained using the estimate $y^2\leq \frac{x^2+a^2}{4}$, which holds for every point $(x,y) \in \Omega_a$ (see Remark \ref{stima su tutto il dominio}). Finally, we obtain (7) by combining (6) with the fact that $v_a(x,0)=0$.
\end{proof}

\begin{oss}\label{vertical segment}
Using the notation in Definition \ref{ImmersionF}, we have that for any $x\in\mathbb R$ $F_a(x,0)=(0,0,x)$. This can be obtain by directly integrating equation (\ref{dxF}), knowing that $v_a(x,0)=0$.

\end{oss}

Let us further fix the following notation for points on the boundary of $\Omega_a$.
\begin{defin}
Given fixed values for $x_0\in\mathbb R$ and $a\in (0,1/2]$, we will denote by $y_{x_0,a}$ the $y$ value of the point of intersection between the line $x=x_0$ and the boundary of $\Omega_a$ in the upper-half plane. Likewise, we will denote by $-y_{x_0,a}$ the $y$ value of the point of intersection between $x=x_0$ and the boundary in the lower-half plane.

In other words:
\[
    y_{x_0,a}=\left\{\begin{array}{lr}
         \frac{(x_0^2+a^2)^{3/4}}{2}\text{ if }\vert x_0\vert \le x_A\,,\\
         \frac{(x_0^2+a^2)^{1/2}}{2}\text{ if }\vert x_0\vert \ge x_B\,,\\
        \xi_1(x_0)\text{ if } x_A\le x_0\le x_B\,, \\
        \xi_2(x_0)\text{ if } -x_B\le x_0\le -x_A
        \end{array}\right.\,,
  \]
and
\[
    -y_{x_0,a}=\left\{\begin{array}{lr}
         -\frac{(x_0^2+a^2)^{3/4}}{2}\text{ if }\vert x_0\vert \le x_A\,,\\
         -\frac{(x_0^2+a^2)^{1/2}}{2}\text{ if }\vert x_0\vert \ge x_B\,,\\
        -\xi_1(x_0)\text{ if } x_A\le x_0\le x_B\,, \\
        -\xi_2(x_0)\text{ if } -x_B\le x_0\le -x_A
        \end{array}\right.\,.
  \]
\end{defin}
\begin{oss}\label{bound su yxa}
As already observed in Remark \ref{stima su tutto il dominio}, we have $\Omega_a\subset \mathcal D^2_a$. This implies in particular that for all $x\in \mathbb R$ and $a\in (0,1/2]$ we have $|y_{x,a}|\le \frac{(x^2+a^2)^{1/2}}{2}$.
\end{oss}

\section{Lower bound: properly embedded in an infinite cylinder}

In this section we want to show that all the minimal disks $\Sigma_a=F_a(\Omega_a)$ are properly embedded in a fixed infinite cylinder around the $x_3$-axis in $\mathbb R^3$. To do this, following \cite{cm28}, we analyse the intersection of $\Sigma_a$ with the horizontal planes $\lbrace x_3=x\rbrace$ for varying $x\in \mathbb R$. We will see that  there exists some $R_0>0$ such that for all $a\in (0,\frac{1}{2}]$ and for all $x \in \mathbb R$, the intersection $\Sigma_a \cap \{x_3=x\}$ is a smooth curve in the plane $\{x_3=x\}$ going through the base point $F_a(x,0)=(0,0,x)$ (see Remark \ref{vertical segment}), whose restriction to the disk $D_{R_0}((0,0,x))$ is properly embedded in $D_{R_0}((0,0,x))$.\par

The following lemmas are the extension of Lemma 3 in \cite{cm28} to our unbounded domains. Let us fix some notations. We will consider the curves
\begin{align}\label{definition delle curve gamma}
    \gamma_{x,a}:[-y_{x,a},y_{x,a}]\to \{x_3=x\}, \gamma_{x,a}(y):=F_a(x,y)
\end{align}

and use the notation $\gamma'_{x,a}(y)=\frac{\partial F_a}{\partial y}(x,y)$. Using equation (\ref{dyF}) on the choice of Weierstrass data made in Definition \ref{ChoiceWeierstrassData}, this vector can easily be computed to be 
$$\gamma'_{x,a}(y)=(\cosh(v_a(x,y)\sin(u_a(x,y),-\cosh(v_a(x,y)\cos(u_a(x,y),0) $$
so in particular we have that
$$\left|\left| \gamma'_{x,a}(y) \right|\right|=\cosh(v_a(x,y)) \quad \textrm{and} \quad \left|\left| \gamma'_{x,a}(0) \right|\right|=1\,.$$

\begin{lemma}\label{facts about the curves}
Using the notations above, the following holds:
\begin{enumerate}
\item $x_3(F_a(x,y))=x$, that is $\gamma_{x,a}=\Sigma_a \cap \{x_3=x\}$,
\item $\cos(\gamma'_{x,a}(y),\gamma'_{x,a}(0))=\cos(u_a(x,y)-u_a(x,0))$,
\item the curve $\gamma_{x,a}$ is a graph over the line with direction $\gamma'_{x,a}(0)$, contained in a sector of angle $\theta_0\leq \frac{1}{x_A}=\frac{6\pi}{4+\pi}$ (see equation (\ref{def xA}) for the definition of $x_A$).
\end{enumerate}
\end{lemma}
\begin{proof}
Statement (1) follows directly from  equation (\ref{ImmersionF}), as $\phi=dz$. Let us prove the other points. Notice that when $|x|\leq x_A$ they are proved in Lemma 3 in \cite{cm28}, so let us focus on the case  $|x|> x_A$.
To prove (2), following \cite{cm28}, we consider
\begin{align*}
\langle \gamma'_{x,a}(y),\gamma'_{x,a}(0)\rangle = \left|\left| \gamma'_{x,a}(y) \right|\right|\left|\left| \gamma'_{x,a}(0) \right|\right|\cos(\gamma'_{x,a}(y),\gamma'_{x,a}(0))
\end{align*}
so that by the above remarks we get 
\begin{align*}
\langle \gamma'_{x,a}(y),\gamma'_{x,a}(0)\rangle =\cosh(v_a(x,y))\cos(\gamma'_{x,a}(y),\gamma'_{x,a}(0))
\end{align*}
On the other hand, equation (\ref{dyF}) can also be used to obtain directly that
    \begin{align*}
                 \langle \gamma'_{x,a}(y),\gamma'_{x,a}(0)\rangle = \cosh(v_a(x,y))\cos(u_a(x,y)-u_a(x,0))
            \end{align*}
which proves (2).\par

To obtain (3) we are going to integrate the inequalities obtained in Lemma \ref{easybounds} and use the fact  that throughout $\Omega_a$ we have $y^2\leq \frac{x^2+a^2}{4}$ (see Remark \ref{bound su yxa}):
$$|u_a(x,y)-u_a(x,0)|\leq \int_0^y| \partial_t u_a(x,t) |dt \leq \dfrac{4|x|}{(x^2+a^2)^2}\frac{y^2}{2} \leq \dfrac{|x|}{2(x^2+a^2)}\,. $$

Let us fix an angle $\theta \in (2, \pi)$. From the last estimate and (2) we get that the absolute value of the angle between $\gamma'_{x,a}(y)$ and $\gamma'_{x,a}(0)$ is at most $\frac{\theta}{2}$ if and only if $|x|>\frac{1}{\theta}$.
We are looking at the case  $|x|>x_A$, so we get that the angle is at most $\frac{1}{2x_A}$, so that all the tangent vectors $\gamma'_{x,a}(y)$ live in a sector of angle at most $\frac{1}{x_A}=\frac{6\pi}{4+\pi}<\pi$ in the direction of $\gamma'_{x,a}(0)$. In particular $\gamma_{x,a}$ is a graph
over the line in that direction. But by continuity this actually implies that the whole curve is contained in the same sector.
\end{proof}

\begin{oss}\label{reason for xA}
For $|x|\leq x_A < \frac{1}{2}$, equation (2.11) in \cite{cm28} proves that $\gamma_{x,a}$ is contained in a sector of angle $2$. In our extension to $|x|>x_A$, we get a sector of angle $\theta_0=\frac{1}{x_A}=\frac{6\pi}{4+\pi}$ (see equation (\ref{def xA}) for the definition of $x_A$). The reason for this particular choice is as follows: in order to get a convex sector we need $0<\theta_0=\frac{1}{x_A}<\pi$, i.e. $\frac{1}{\pi}<x_A$, but in order to get a connected domain $\Omega_a$ we also need $x_A<\frac{1}{2}$. 

The first third of the interval $[\frac{1}{\pi},\frac{1}{2}]$ turns out to be a geometrically convenient choice to satisfy both conditions, and this is the main reason for the specific choice of $x_A$ we made above. Any other choice for the value of $x_A$ in the interval $(\frac{1}{\pi},\frac{1}{2})$ would work.
\end{oss}

In Lemma 3 of \cite{cm28}, it is also proved that when $|x|\leq 1/2$, the curves $\gamma_{x,a}$ have their endpoints outside the disks $D_{r_0}((0,0,x))$, where $r_0>0$ does not depend on $a$ or $x$. We are now going to prove that the same is still true for arbitrarily large values of $|x|$.

\begin{lemma}\label{properly embedded}
Using the notations above, there exists
$ r_0'>0$ such that for any $|x|>x_A$, and for any $ a \in (0,\frac{1}{2}]$, we have $||\gamma_{x,a}(y_{x,a})-\gamma_{x,a}(0)||>r_0'$
\end{lemma} 
\begin{proof}
 For any $y\in [-y_{x,a},y_{x,a}]$ we have the following equality
\begin{equation}\label{norm_excursion}
||\gamma_{x,a}(y)-\gamma_{x,a}(0)||=\dfrac{\langle \gamma_{x,a}(y)-\gamma_{x,a}(0) , \gamma'_{x,a}(0)\rangle}{||\gamma'_{x,a}(0)||\cos ( \gamma_{x,a}(y)-\gamma_{x,a}(0),\gamma'_{x,a}(0) )}
\end{equation}
and we want to find a lower bound for this quantity which is independent of $a$ and $x$.
As observed above, $||\gamma'_{x,a}(0)||=1$, and of course the cosine term is bounded. So it is enough to bound the numerator of the right hand side. Reasoning as in the previous lemma, we have the following:
\begin{align}
   \label{TO BE USED IN SECTION 9 1} \langle \gamma_{x,a}(y)-\gamma_{x,a}(0) ,& \gamma'_{x,a}(0)\rangle = \int _0^y \langle \gamma'_{x,a}(t) ,\gamma'_{x,a}(0) \rangle dt \\=&\int_0^y \cosh v_a(x,t)\cos(u_a(x,t)-u_a(x,0)) dt\,. \label{TO BE USED IN SECTION 9 2}
\end{align}
From Lemma \ref{facts about the curves} we get that the absolute value of the argument of the cosine is at most $\frac{\theta_0}{2}=\frac{1}{2x_A}=\frac{3\pi}({4+\pi})$, so we can get a lower bound for the cosine. 

Moreover from Lemma \ref{easybounds} we have $v_a(x,y)\geq 0$ when $y\ge 0$. This implies that the integrand function is non-negative, and allows us to get
\begin{align}
\langle \gamma_{x,a}(y)-\gamma_{x,a}(0) ,& \gamma'_{x,a}(0)\rangle  \geq \cos\left( \theta_0 \right) \int_{y_{x,a}/2}^{y_{x,a}} \cosh v_a(x,t) dt
\end{align}
so it is now enough to provide a lower bound for $v_a$. Integrating the estimate in point (6) of Lemma \ref{easybounds}, one easily gets that
\begin{align}
    \min_{y_{x,a}/2\leq y \leq y_{x,a}} |v_a(x,y)|=\left|v_a\left(x,\frac{y_{x,a}}{2}\right)\right|\geq \frac{3}{32\sqrt{x^2+a^2}}
\end{align}
In the end we get
\begin{align*}
    \langle \gamma_{x,a}(y)-\gamma_{x,a}(0) ,& \gamma'_{x,a}(0)\rangle  \geq \cos(\theta_0) \cosh \left(\frac{3}{32\sqrt{x^2+a^2}} \right) \frac{\sqrt{x^2+a^2}}{4}\geq
\end{align*}
\begin{align*}
   \geq \frac{\cos(\theta_0)}{4}|x| \geq
    \frac{\cos(\theta_0)}{4}x_A
    \geq \frac{\cos(\theta_0)}{4\pi}\,.
\end{align*}

\end{proof}

We conclude this section with the following statement.
\begin{prop}\label{multivalued graphs}
There exists $R_0>0$ such that, for any $a \in (0,\frac{1}{2}]$, the surface $\Sigma_a\cap \{x_1^2+x_2^2\leq R_0^2, x_3 \in \mathbb R\}$ is a properly embedded minimal disk. Moreover, we have that $\lbrace 0<x_1^2+x_2^2<R_0^2,\,x_3\in\mathbb R\rbrace\cap F_a(\Omega_a)=\Tilde{\Sigma}_{1,a}\cup\Tilde{\Sigma}_{2,a}$ for multi-valued graphs $\Tilde{\Sigma}_{1,a}$ and $\Tilde{\Sigma}_{2,a}$ over $D_{R_0}(0)\setminus\lbrace0\rbrace$.
\end{prop}
\begin{proof}
By Lemma \ref{facts about the curves} we have that $F_a$ is an embedding. From Corollary 1 in \cite{cm28}, it follows that there exists a $r_0>0$ such that for all $|x|\leq x_A$ and $a\in (0,\frac{1}{2}]$ we have that $\Sigma_a\cap \{x_3=x\}$ intersects $D_{r_0}((0,0,x))$ in a properly embedded arc. A similar statement holds when $|x|>x_A$ for a possibly different $r_0'>0$ thanks to Lemma \ref{properly embedded}. It is then enough to take $R_0=\min\{r_0,r_0'\}$.

Furthermore, if we consider equation ($\ref{unit normal 1})$, we get that $F_a$ is \textit{vertical}, that is $\langle\mathbf{n},(0,0,1)\rangle=0$ when $\vert g_a\vert=1$. However, $\vert g_a(x,y)\vert=1$ exactly when $y=0$, hence by Remark \ref{vertical segment}, we know that the image is graphical away from the $x_3$-axis. Combining this with Lemma \ref{properly embedded} for $\vert x\vert\ge x_A$ and equation (2.8) in \cite{cm28} for $|x|\leq x_A$, we finally get our result.
\end{proof}

\section{Upper bound: not properly embedded in an exhaustion of $\mathbb{R}^3$}

We are left to prove that our minimal disks $\Sigma_a$ are not properly embedded in any growing sequence of open sets that exhausts $\mathbb R^3$. 

In order to do so, we will show that there exists at least a value $x\neq 0$ for which the maximal excursion of all helicoids on the horizontal slice $\lbrace x_3=x\rbrace$ is bounded from above by a constant $C_x$ depending only on $x$. This would then imply that one can find an infinite cylinder \textit{wide enough} (for example, we could take $\lbrace x_1^2+x^2_2 < 2C_x,\,x_3\in\mathbb R\rbrace$) in which the minimal disks $\Sigma_a$ are not properly embedded. As we will see in Remark \ref{per ogni x non 0} this is actually the case for all $x \neq 0$.

\begin{lemma}\label{Upper bound of v(x,y_ax)}
Given an arbitrary value $x\in\mathbb R\setminus\lbrace 0\rbrace$, we have
\begin{align}
    \max_{0\le\vert y\vert\le y_{a,x}}\vert v_a(x,y_{x,a})\vert\le M_x\,,
\end{align}
where $M_x$ is a constant that depends only on $x$.
\end{lemma}
\begin{proof}
By construction, as already pointed out in Remark \ref{bound su yxa}, we know that $ y^2\le \frac{1}{4}(x^2+a^2)$ for any $(x,y)\in\Omega_a$ and for any $a\in (0,1/2]$, hence from Lemma \ref{easybounds} we get:
\begin{align*}
    \partial_yv_a(x,y)=&\frac{x^2+a^2-y^2}{(x^2+a^2-y^2)^2+4x^2y^2}\le \frac{x^2+a^2}{(x^2+a^2-(x^2+a^2)/4)^2}=\frac{16}{9}\frac{1}{x^2+a^2}\,.
\end{align*}

Applying this estimate (and knowing $v_a(x,y)$ is an increasing function in $y$, see Lemma \ref{easybounds}), we then obtain:
\begin{align}
    \max_{0\le\vert y\vert\le y_{x,a}}\vert v_a(x,y)\vert&\le\int_0^{y_{x,a}}\vert\partial_y v_a(x,y)\vert dy=\int_0^{y_{a,x}}\partial_yv_a(x,y)dy\\&\le\frac{16}{9}\frac{1}{x^2+a^2}\int_{0}^{y_{a,x}}dy=\frac{16}{9}\frac{1}{x^2+a^2}\frac{(x^2+a^2)^{1/2}}{2}\\&=\frac{8}{9}\frac{1}{\sqrt{x^2+a^2}}\le \frac{8}{9\vert x\vert}=:M_x\;\;\forall\,a\in(0,1/2]\,.
    \end{align}
    \end{proof}
    
    \begin{lemma}\label{bound from above on scalar product} Given an arbitrary value $x\in\mathbb R\setminus\lbrace 0\rbrace$, we also have:
    \begin{align}
       \big\vert \langle \gamma_{x,a}(y_{x,a})-\gamma_{x,a}(0), \gamma'_{x,a}(0)\rangle\big\vert\le N_x\,,
    \end{align}
    where $N_x$ is a constant that depends only on $x$.
    \end{lemma}
    \begin{proof}
    Using the equation (\ref{TO BE USED IN SECTION 9 1}, \ref{TO BE USED IN SECTION 9 2}), we get:
    \begin{align}
        \langle \gamma_{x,a}(y_{x,a})-\gamma_{x,a}(0) , \gamma'_{x,a}(0)\rangle=\int_0^{y_{x,a}}\cosh v_a(x,t)\cos(u_a(x,t)-u_a(x,0)) dt\,.
    \end{align}
    
    Therefore, since cosine is a bounded function, and using the estimate provided by Lemma \ref{Upper bound of v(x,y_ax)}, we obtain:
    \begin{align}
        \big\vert\langle \gamma_{x,a}(y_{x,a})-\gamma_{x,a}(0), \gamma'_{x,a}(0)\rangle \big\vert\le\int_0^{y_{x,a}}\vert\cosh(v_a(x,y))\vert dy\\
        =\int_0^{y_{x,a}}\cosh(v_a(x,y)) dy\le\cosh(M_x)\int_0^{y_{x,a}} dy\\ \le\cosh(M_x)\cdot\frac{\sqrt{x^2+a^2}}{2}\le \cosh(M_x)\cdot \frac{\sqrt{x^2+1/4}}{2}&=:N_x\,.
    \end{align}
    \end{proof}
    \begin{oss}\label{per ogni x non 0}
    Let us study the behaviour of $N_x$. As already pointed out before, $M_x=C\cdot 1/\vert x\vert$, hence
    \begin{align}
        N_x\approx C_2 \cosh\bigg(\frac{1}{\vert x\vert}\bigg)\cdot(\vert x\vert +C_2)\,,
    \end{align}
    therefore $N_x\to\infty$ both when $\vert x\vert\to 0$, and when $\vert x\vert\to\infty$. On the other hand, for any $x$ for which $\vert x\vert\in (0,\infty)$, we have that $N_x$ is bounded.
    \end{oss}
    \begin{prop}The minimal disks $\Sigma_a$ cannot be properly embedded in any growing sequence of open sets in $\mathbb R^3$ that exhausts $\mathbb R^3$.
    \end{prop}
    \begin{proof}
    In order to prove this, it is sufficient to show that there exists at least one value of $x\in\mathbb R$ for which the maximal excursion of all the helicoids $\Sigma_a$ on the horizontal slice $\lbrace x_3=x\rbrace$ is bounded from above by a constant that depends only on $x$.
    
    In order to do so, we will be combining the results provided in Lemma \ref{bound from above on scalar product} and Lemma \ref{facts about the curves}, so we finally obtain:
    \begin{align*}
        \Vert \gamma_{x,a}(y_{x,a})-\gamma_{x,a}(0)\Vert =\frac{\langle \gamma_{x,a}(y_{x,a})-\gamma_{x,a}(0),\gamma_{x,a}'(0) \rangle}{\cos \left( \gamma_{x,a}(y_{x,a})-\gamma_{x,a}(0),\gamma_{x,a}'(0) \right)}\le C\,N_x=:C_x\,.
    \end{align*}
   
    \end{proof}
\section{Proof of Theorem \ref{main thm}}
This result is analogous to Theorem 1 in Colding and Minicozzi's paper
\cite{cm28}. In order for this paper to be self-contained, we report the proof below.

\begin{proof}[Proof of Theorem \ref{main thm}]
Up to rescaling, it is sufficient to show that there exists a sequence $\Sigma_i\subset\lbrace x_1^2+x_2^2\le R^2,\,x_3\in\mathbb R\rbrace$, for some $R>0$. Lemma \ref{facts about the curves} gives minimal embeddings $F_a:\Omega_a\to\mathbb R^3$ with $F_a(x,0)=(0,0,x)$ for any $x\in\mathbb R$ (see Remark \ref{vertical segment}), so by Proposition \ref{multivalued graphs} (3) holds for any $R\le R_0$.

Let us take $R=\frac{R_0}{2}$ and $\Sigma_i=\lbrace x_1^2+x_2^2\le R,\,x_3\in\mathbb R\rbrace\cap F_{a_i}(\Omega_{a_i})$, where the sequence $a_i$ is to be determined.

To get (1), simply notice that, by equation in point 2. of Lemma \ref{easybounds}, we get $\vert\mathbf{K}_a\vert(0)=a^{-4}\to\infty$ as $a\to 0$.

Still applying the same equation in point 2. of Lemma \ref{easybounds}, we get that for all $\delta >0$, $\sup_a\sup_{\lbrace\vert x\vert\ge\delta\rbrace\cap\Omega_a}\vert\mathbf{K}_a\vert<\infty$ . Combine this with (3) and Heinz's curvature estimate for minimal graphs (see equation (11.7) in \cite{osserman2002survey}), and we get (2).

In order to get (4), we use Lemma 2 in \cite{cm28} to choose $a_i\to 0$ so the mappings $F_{a_i}$ converge uniformly in $C^2$ on compact subsets to $F_0:\Omega_0\to\mathbb R^3$. Hence, by Lemma \ref{facts about the curves} and Proposition \ref{multivalued graphs}, the $\Sigma_i\setminus\lbrace x_3=0\rbrace$ converge to two embedded minimal disks $\Sigma^\pm\subset F_0(\Omega_0^\pm)$ with $\Sigma^\pm\setminus\lbrace x_3\text{-axis}\rbrace=\Sigma_1^\pm\cup\Sigma_2^\pm$ for multi-valued graphs $\Sigma_i^\pm$.

To complete the proof, one should notice that in a small enough neighbourhood of the plane $\{x_3=0\}$, our minimal disks $\Sigma_i$ coincide with those constructed in \cite{cm28}. Hence each graph $\Sigma_i^{\pm}$ is $\infty$-valued (and hence spirals into the horizontal plane $\lbrace x_3=0\rbrace$), completing point (4).


\end{proof}


\printbibliography

\end{document}